\documentclass[11pt,reqno]{amsart}

\usepackage[english]{babel}
\usepackage[T1]{fontenc}
\usepackage[utf8]{inputenc}
\usepackage{csquotes}

\usepackage[title,titletoc,page]{appendix}
\usepackage[a4paper,left=3cm,right=3cm,top=2cm,bottom=4cm,bindingoffset=5mm]{geometry}
\usepackage[colorlinks=true,linkcolor=blue,linktocpage=true]{hyperref}
\usepackage[all]{hypcap}
\usepackage{setspace}
\usepackage{xcolor}

\usepackage{aligned-overset}
\usepackage{amsfonts}
\usepackage{amsmath}
\usepackage{amssymb}
\usepackage{amsthm}
\usepackage{dsfont}
\usepackage{float}
\usepackage{mathrsfs}
\usepackage{mathtools}

\usepackage{thmtools}
\usepackage[shortlabels]{enumitem}

\usepackage{graphicx}

\usepackage{mathscinet}
\usepackage[backend=biber,sorting=nty,natbib=true, isbn=false,doi=false,url=false,maxbibnames=5,style=ieee]{biblatex}
\usepackage{version}
\usepackage[obeyFinal]{todonotes}
\usepackage{cancel}
\DeclareMathOperator{\R}{\mathbb{R}}
\DeclareMathOperator{\N}{\mathbb{N}}

\DeclareMathOperator{\E}{\mathbb{E}}

\DeclareMathOperator{\eps}{\varepsilon}

\numberwithin{equation}{section}

\newtheorem{theorem}{Theorem}[section]
\newtheorem{lemma}[theorem]{Lemma}

\allowdisplaybreaks

\theoremstyle{definition}
\newtheorem{definition}[theorem]{Definition}

\newtheorem{remark}[theorem]{Remark}

\usepackage{bbm}
\usepackage[normalem]{ulem}
\usepackage{cancel}

\newcommand{\cA}{\mathcal{A}}

\newcommand{\bE}{\mathbb{E}}

\newcommand{\bR}{\mathbb{R}}

\renewcommand{\P}{\mathbb{P}}
\newcommand{\dd}{\mathrm{d}}
\newcommand{\UCb}{\mathrm{UC}_{\rm b}}
\def\eps{\varepsilon}

\newcommand{\cbar}{\overline{c}}
\newcommand{\cbarbar}{\overline{\overline{c}}}

\title[Viscosity Solutions to Abstract Cauchy Problems]{Existence of Viscosity Solutions to Abstract Cauchy Problems via Nonlinear Semigroups}

\author{Fabian Fuchs}
\address{Center for Mathematical Economics, Bielefeld University, Germany}
\email{fabian.fuchs@uni-bielefeld.de}

\author{Max Nendel}
\address{Department of Statistics and Actuarial Science, University of Waterloo, Canada}
\email{mnendel@uwaterloo.ca}

\thanks{This work was funded by the Deutsche Forschungsgemeinschaft (DFG, German Research Foundation) -- SFB 1283/2 2021 -- 317210226. Part of this research was conducted during a stay of the second named author at the Sydney Mathematical Research Institute (SMRI) as part of the International Visitor Program 2024/2025. The second named author is grateful for the financial support and hospitality of the SMRI}

\date{\today}

\begin{document}

\begin{abstract}
    In this work, we provide conditions for nonlinear monotone semigroups on locally convex vector lattices to give rise to a generalized notion of viscosity solutions to a related nonlinear partial differential equation.\
    The semigroup needs to satisfy a convexity estimate, so called $K$-convexity, w.r.t.\ another family of operators, defined on a potentially larger locally convex vector lattice. We then show that, under mild continuity requirements on the bounding family of operators, the semigroup yields viscosity solutions to the abstract Cauchy problem given in terms of its generator in the larger locally convex vector lattice.\ We apply our results to drift control problems for infinite-dimensional L\'evy processes and robust optimal control problems for infinite-dimensional Ornstein-Uhlenbeck processes.\smallskip \ 
\end{abstract}

\maketitle


\section{Introduction}

In this work, we discuss under which conditions classes of nonlinear monotone semigroups, 
defined on a locally convex vector lattice $(X, \tau_X)$, give rise to viscosity solutions to the abstract Cauchy problem
\begin{align}
    u'(t) &= Lu(t), \label{eq:intro:cauchy} \\ 
    u(0) &= u_0, \nonumber
\end{align}
where $L$ is the infinitesimal generator of the semigroup $S$ and $u_0 \in X$.
We require the semigroup to satisfy a convexity-like estimate, which we call $K$-convexity, cf.\ Definition \ref{def:k-convex}, where the family of operators $K$ can be defined on a potentially larger locally convex vector lattice $(Y, \tau_Y)$ with a coarser topology $\tau_Y$ in the sense that the identity $X\to Y$ is continuous.\
Replacing convexity by $K$-convexity for the semigroup $S$ broadens the class of nonlinear semigroups we are able to treat, notably including semigroups with $\inf$-$\sup$-type nonlinearities.\ Additionally, this relaxation allows us to transfer continuity requirements from the semigroup $S$ to the family of operators $K$ and consider the semigroup's generator in the larger space $Y$ with its coarser topology $\tau_Y$ rather than in $(X, \tau_X)$.\ 

The semigroups and partial differential equations treated in this paper mainly arise from optimal control theory.\ In a finite-dimensional setting, the monographs \cite{FlSo06,Ph09,YoZh99} introduce a general theory and relate it to applications in Economics and Finance. For an introduction with applications in the infinite-dimensional context, see \cite{FaGoSw17}. It is of note that \cite{FlSo06} also provides a semigroup perspective, which, however, is not explored in depth there.

Other prominent examples that give rise to semigroups which can be treated in the presented framework are related to pricing under uncertainty. 
In the financial literature, such problems are often studied using nonlinear expectations. The framework of $G$-Brownian Motion or, more generally $G$-L\'evy processes, developed in \cite{Pe19,hu.2021, NeNu2017}, is used to study problems related to asset pricing under model uncertainty. Similarly, the works \cite{Ho16,DeKuNe20,Ku19} study L\'evy processes on sublinear expectation spaces and their relation to Hamilton-Jacobi-Bellman equations from the angle of semigroup theory.\
Volatility uncertainty arises in many financial applications and has, for example, been studied by \cite{AvLePa95} in the context of derivative pricing, by \cite{EpJi13} for utility maximization problems, and by \cite{Vo14} in relation to market completeness.

All of the previously mentioned examples lead to convex operators.\ However, in the context of controlled processes under uncertainty, also non-convex operators of $\inf$-$\sup$-type naturally appear, cf.\ \cite{NeSe24,NeSe24b} for examples related to Markov Decision Processes and \cite{BaEcKu21,BaNePa23} for works on robust optimization problems, as well as differential games, cf.\ \cite{Is65,FrQu92,Ga08,EvSo84}.

To describe solutions to Cauchy problems of type \eqref{eq:intro:cauchy} in our abstract setting, we introduce a generalized notion of viscosity solutions, making the dependence of the test functions on the domain $D(L)$ of the generator and a set $M$ of continuous linear functionals, where test functions are required to touch the solution, explicit. The notion of viscosity solutions was introduced by \cite{CrLi83} as a solution concept for Hamilton-Jacobi equations, for which no classical solutions exist or uniqueness of weak solutions fails. The name \emph{viscosity solution} stems from the vanishing viscosity method, used by the authors to show the existence of such solutions.
Most modern existence proofs, however, either use versions of Perron's method, see e.g.\ \cite{CIL92,Ba94}, or show the viscosity property of a candidate solution directly. In optimal control theory, this candidate solution is the value function of the related optimal control problem, cf.\ \cite{FlSo06,YoZh99,Ph09,FaGoSw17}.
The approach in this work abstracts the methodology from optimal control and transfers the question of existence of viscosity solutions to the existence of semigroups with suitable properties, following the paradigm of classical semigroup theory.
The works \cite{GoNeRo24, NeRo21, DeKuNe20} take a similar perspective but have largely disjoint frameworks as each of them focuses on different applications. In the present paper, we provide a unified framework and extend the above results on the existence of viscosity solutions to arbitrary locally convex vector lattices and $K$-convex operators.

The generality of the setting in this work is partially motivated by a rising interest in infinite-dimensional problems, cf.\ \cite{BeLi24, BaChEkQiTa25, SoTiZh24} for recent examples. These problems are explicitly covered by our setup, which we illustrate on two examples, see Section \ref{sec:examples}.
We first treat a drift-controlled infinite-dimensional L\'evy process and then a drift-controlled Ornstein-Uhlenbeck process under volatility uncertainty.
The value functions of both examples are nonlinear due to the optimization over controls. The presence of volatility uncertainty in the second example leads to a non-convex value function of $\inf$-$\sup$-type. Using our main result, we characterize these value functions as viscosity solutions to the abstract Cauchy problem given in terms of the generator of the value function.

The rest of the paper is organized as follows. In Section \ref{sec:setup}, we introduce our setup and main concepts such as $K$-convexity and our notion of viscosity solution. Moreover, we show that, for $K$-convex semigroups, continuity requirements can be transferred from the semigroup to the family $K$. In Section \ref{sec:main_result}, we state and prove the main result of this work, Theorem \ref{thm:existence}. Section \ref{sec:examples} contains two applications that  illustrate the use of the main theorem in the context of controlled infinite-dimensional L\'evy processes and controlled infinite-dimensional Ornstein-Uhlenbeck processes under uncertainty.

\section{Setup and preliminaries}\label{sec:setup}

Throughout, let $(Y,\tau_Y)$ be a locally convex vector lattice, i.e., $Y$ is a vector lattice, endowed with a locally convex topology $\tau_Y$ induced by a set of lattice seminorms. Recall that a seminorm $p\colon Y\to [0,\infty)$ is called a lattice seminorm if, for all $u,v\in Y$,
\begin{equation}\label{eq.def.lattice-semi}
    |u|\leq |v|\quad\text{implies}\quad p(u) \leq  p(v),
\end{equation}
where $|u|:=u\vee(-u)$ for $u\in Y$. 

As usual, $Y_+ \coloneqq \{u \in Y \,|\, u \geq 0\}$ denotes the set of all positive elements of $Y$ and $Y'$ the topological dual space of $Y$, i.e., the set of all continuous, linear functionals $Y \to \bR$. Analogously, we denote the set of all positive dual elements of $Y$ as 
\begin{equation*}
    Y'_+ \coloneqq \big\{\mu \in Y' \,\big|\, \mu u \geq 0\text{ for all } u \in Y_+ \big\}.
\end{equation*}
Moreover, we consider a sublattice $X\subset Y$, endowed with a locally convex topology $\tau_X$ induced by a set of lattice seminorms such that the identity $X\to Y$ is continuous.

\begin{remark}\label{rem.mixed.topology}
Typical choices for the spaces $X$ and $Y$ in the context of classical viscosity theory are $Y$ consisting of all functions $\R^d\to \R$ with $d\in \N$ endowed with the pointwise order and the topology of pointwise convergence. In this case, $X$ can, for example, be chosen as the space ${\rm C}_{\rm b}(\R^d)$ of all bounded continuous functions $\R^d\to \R$, endowed with the so-called mixed topology, cf.\ \cite{Se72,Wi61,GK01,GoNeRo24}. Throughout, we are only interested in sequential properties related to the topologies $\tau_X$ and $\tau_Y$. In view of this fact, recall that a sequence $(f_n)_{n\in \N}\subset {\rm C}_{\rm b}(\R^d)$ converges to a function $f\in {\rm C}_{\rm b}(\R^d)$ w.r.t.\ the mixed topology if and only if the sequence $(f_n)_{n\in \N}$ is uniformly bounded and converges to $f$ uniformly on compacts, i.e.,
\[
\sup_{n\in \N}\|f_n\|_\infty<\infty\quad \text{and}\quad \lim_{n\to \infty}\sup_{x\in C} |f_n(x)-f(x)|=0
\]
for all compact sets $C\subset \R^d$, where $\|\cdot\|_\infty$ denotes the supremum norm on ${\rm C}_{\rm b}(\R^d)$.

\end{remark}

We now introduce the central object of our study.

\begin{definition}[Monotone semigroup]\label{def:semigroup}
   A family $S = (S(t))_{t\geq 0}$ of operators $X \to X$ is called a \emph{monotone semigroup on $X$} if it satisfies the following properties:
    \begin{enumerate}[(i)]
        \item $S(0) u = u$ for all $u \in X$.
        \item $S(t) S(s) u = S(t+s) u$ for all $t, s \geq 0$ and $u \in X$.
        \item For all $t \geq 0$, the operator $S(t)$ is monotone, i.e., $S(t) u \leq S(t) v$ for all $u,v\in X$ with $u \leq v$.
    \end{enumerate}
\end{definition}

Observe that the definition of a monotone semigroup only consists of algebraic and order-theoretic properties.

\begin{definition}[$K$-convexity]\label{def:k-convex}
    Let $K = (K(t))_{t\geq 0}$ be a family of operators $X \to Y$. Then, a monotone semigroup $S = (S(t))_{t \geq 0}$ is called \emph{$K$-convex} if, for all $u, v \in X$ and $\lambda \in [0, 1]$,
    \begin{equation*}
        S(t)\left(\lambda u + (1-\lambda) v \right) \leq \lambda S(t) u + (1-\lambda) K(t) v.
    \end{equation*}
\end{definition}

Before we continue our discussion, we briefly illuminate the concept of $K$-convexity in the following remark.

\begin{remark}\label{rem.Kconvex}
Let $S=(S(t))_{t\geq 0}$ be a monotone semigroup.
 \begin{enumerate}[a)]
 \item[a)] If $S$ is convex, i.e.,
\[
 S(t)\big(\lambda u+(1-\lambda)v\big)\leq \lambda S(t)u+(1-\lambda)S(t)v
\]
for all $t\geq0$, $u,v\in X$, and $\lambda\in [0,1]$, then $S$ is $S$-convex.
\item[b)]  If $S$ is $K$-convex and there exists some family $L=(L(t))_{t\geq0}$ of operators $X\to Y$ with $L(t)u \geq K(t)u$ for all $t\geq 0$ and $u\in X$, then $S$ is also $L$-convex.
\item[c)] Assume that $S$ is $K$-convex.\ Then, choosing $\lambda = 0$ in the definition of $K$-convexity, we find that the family $K$ is an upper bound of $S$, i.e., for $t\geq 0$ and $u \in X$, we have
    \begin{equation}\label{eq.upperbound}
        S(t)u \leq K(t)u.
    \end{equation}
    On the other hand,
    \begin{equation}\label{eq.remark.Kconvex.c}
    2S(t)0-S(t)u\leq K(t)(-u)
    \end{equation}
    for all $t\geq0$ and $u\in X$. Indeed, let $t\geq 0$ and $u\in X$. Then,
   \[
    S(t)0=S(t)\Big(\tfrac12 u+\tfrac12 (-u)\Big)\leq \tfrac12 S(t)u+\tfrac12 K(t)(-u).
   \]
   \item[d)] Assume that $S$ is $K$-convex. Then, by part c),
   \[
    S\big(\lambda u+(1-\lambda)v\big)\leq \lambda S(t) u+(1-\lambda)K(t)v\leq \lambda K(t)u+(1-\lambda)K(t)v
   \]
   for all $t\geq0$, $u,v\in X$, and $\lambda\in [0,1]$.
 \end{enumerate}
\end{remark}

\begin{lemma}\label{lem.Kconvex}
Let $S$ be a $K$-convex monotone semigroup with $S(t)0\geq 0$ for all $t\geq0$. Then, for all $t\geq0$ and $u,v\in X$,
 \begin{equation}\label{eq.lem.Kconvex1}
  |S(t)u-v|\leq |K(t)u-v|+|K(t)(-u)+v|. 
\end{equation}
In particular, 
\begin{equation}\label{eq.lem.Kconvex2}
p\big(S(t)u-v\big)\leq p\big(K(t)u-v\big)+p\big(K(t)(-u)+v\big)
\end{equation}
for all $t\geq 0$, $u,v\in X$, and every lattice seminorm $p\colon Y\to [0,\infty)$.
\end{lemma}

\begin{proof}
 First, observe that, by \eqref{eq.upperbound}, 
\[
 S(t)u-v\leq K(t)u-v
\]
for all $t\geq0$ and $u,v\in X$. Moreover, since $S(t)0\geq0$ for all $t\geq0$, by \eqref{eq.remark.Kconvex.c}, it follows that
\[
-\big(S(t)u-v\big)=-S(t)u+v\leq K(t)(-u)+v-2S(t)0\leq K(t)(-u)+v
\]
for all $t\geq 0$ and $u,v\in X$.\ This proves that \eqref{eq.lem.Kconvex1} holds. Now, let $p\colon Y\to [0,\infty)$ be a lattice seminorm. Then, by \eqref{eq.def.lattice-semi}, \eqref{eq.lem.Kconvex1}, and the triangle inequality,
\begin{align*}
 p\big(S(t)u-v\big)=p\big(|S(t)u-v|\big)&\leq p\big(|K(t)u-v|+|K(t)(-u)+v|\big)\\
 &\leq p\big(|K(t)u-v|\big)+p\big(|K(t)(-u)+v|\big)\\
 &=p\big(K(t)u-v\big)+p\big(K(t)(-u)+v\big).
\end{align*}
The proof is complete.
\end{proof}

The following lemma extends a classical estimate for convex operators to the $K$-convex case.

\begin{lemma}\label{lem.difference-quotient}
Let $S$ be a $K$-convex monotone semigroup. Then, for all $t\geq0$, $u,v\in X$, and $h>0$,
\begin{equation}\label{eq.lem.difference-quotient}
\frac{S(t)u -S(t)v}{h}\leq K(t)\bigg(v+\frac{u-v}{h}\bigg)-S(t)v.
\end{equation}
\end{lemma}

\begin{proof}
Let $t\geq0$, $u,v\in X$, and $h>0$. Then, using the $K$-convexity of $S$ as
    \begin{align*}
        S(t)u - S(t)v &= S(t)\Bigg( (1-h)v + h\bigg(v+\frac{u - v}{h} \bigg)\Bigg) - S(t)v\\
        &\leq (1-h)S(t)v + h K(t)\bigg(v+ \frac{u-v}{h}\bigg) - S(t)v\\
        &= h\Bigg(K(t)\bigg(v+ \frac{u-v}{h}\bigg) - S(t)v\Bigg).
    \end{align*}
 Dividing by $h>0$, the claim follows.
\end{proof}

We now turn our focus to continuity properties of the family $K=(K(t))_{t\geq0}$, which, by Lemma \ref{lem.Kconvex}, {imply} continuity properties for the semigroup $S=(S(t))_{t\geq0}$.

\begin{definition}[Strong right-continuity]\label{def.strong-right}
We say that a family $K=(K(t))_{t\geq0}$ of operators $X\to Y$ is \emph{strongly right-continuous} if, for all sequences $(t_n)_{n\in \N}\subset (0,\infty)$ with $t_n\to 0$ as $n\to \infty$ and $(u_n)_{n\in \N} \subset X$ with $u_n\to u_0\in X$ as $n\to \infty$ in the topology $\tau_X$, {we have}
\begin{equation}
    \lim_{n\to \infty} K(t_n)u_n = u_0 \quad\text{in the topology }\tau_Y.
\end{equation}
\end{definition}

\begin{lemma}\label{lem.right-cont-S}
 Let $S$ be a $K$-convex semigroup with strongly right-continuous $K$ and $S(t)0\geq 0$ for all $t\geq 0$. Then, $S$ strongly right-continuous.
\end{lemma}

\begin{proof}
 Let $(t_n)_{n\in \N}\subset (0,\infty)$ with $t_n\to 0$ as $n\to \infty$ and $(u_n)_{n\in \N} \subset X$ with $u_n\to u_0\in X$ as $n\to \infty$ in the topology $\tau_X$. Since $X$ is a locally convex space, hence a topological vector space, also $-u_n\to -u_0$ as $n\to \infty$ in the topology $\tau_X$.\ By Lemma \ref{lem.Kconvex}, we therefore obtain that
 \[
 p\big(S(t_n)u_n-u_0\big)\leq p\big(K(t_n)u_n-u_0\big)+p\big(K(t_n)(-u_n)+u_0\big)\to 0\quad\text{as }n\to \infty
 \]
 for every $\tau_Y$-continuous lattice seminorm $p\colon Y\to [0,\infty)$.
\end{proof}

In the following, we consider the infinitesimal generator of the semigroup $S=(S(t))_{t\geq0}$ in the coarser topology $\tau_Y$. 

\begin{definition}[Infinitesimal generator]\label{def:generator}
    Let $S$ be a monotone semigroup on $X$. The \textit{infinitesimal generator} $L\colon D(L)\subset X\to Y$ of $S$ is defined by
    \begin{align*}
     D(L)&\coloneqq\bigg\{u\in X\,\bigg|\, \lim_{t \downarrow 0} \frac{S(t)u - u}t\in Y \text{ exists w.r.t.\ }\tau_Y\bigg\}\quad\text{and}\\
        Lu &\coloneqq \lim_{t \downarrow 0} \frac{S(t)u - u}{t}\quad\text{in the topology }\tau_Y\text{ for }u\in D(L).
    \end{align*}
\end{definition}

Given an operator $L\colon D\subset X\to Y$, we are interested in solutions to the following abstract differential equation
\begin{equation}\label{eq:abstract_eq}
    u'(t) = Lu(t) \quad \text{for all } t>0.
\end{equation}
In the following, we introduce a solution concept to equations of the type \eqref{eq:abstract_eq}, which generalizes the classical notion of a viscosity solution, cf.\ \cite{CrLi83,CIL92}, to our setup.

\begin{definition}[Viscosity solution]
    Let $L \colon D \subset X \rightarrow Y$ be an operator and $M \subset Y'_+$.
    \begin{enumerate}[(i)]
        \item We say that $u \colon [0, \infty) \rightarrow X$ is a \textit{$D$-$M$-viscosity subsolution} to equation \eqref{eq:abstract_eq} if $u$ is continuous and, for all $t>0$, $\mu \in M$, and every differentiable function $\varphi \colon (0, \infty) \rightarrow X$ with $\varphi(t) \in D$, $\mu\varphi(t) = \mu u(t)$, and $\varphi(s) \geq u(s)$ for all $s >0$, we have
        \begin{equation*}
            \mu\varphi'(t) \leq \mu L \varphi (t).
        \end{equation*}
        
        \item We say that $u \colon [0, \infty) \rightarrow X$ is a \textit{$D$-$M$-viscosity supersolution} to equation \eqref{eq:abstract_eq} if $u$ is continuous and, for all $t>0$, $\mu \in M$, and every differentiable function $\varphi \colon (0, \infty) \rightarrow X$ with $\varphi(t) \in D$, $\mu\varphi(t) = \mu u(t)$, and $\varphi(s) \leq u(s)$ for all $s >0$, we have
        \begin{equation*}
            \mu\varphi'(t) \geq \mu L \varphi (t).
        \end{equation*}

        \item We say $u$ is a \textit{$D$-$M$-viscosity solution} to equation \eqref{eq:abstract_eq} if $u$ is a $D$-$M$-viscosity sub- and supersolution.
    \end{enumerate}
\end{definition}

\begin{remark}
Consider the setup from Remark \ref{rem.mixed.topology}, and let $D:={\rm C}_{\rm b}^2(\R^d)$ be the space of all twice differentiable functions with bounded continuous derivatives. Then, $\varphi\colon (0,\infty)\to {\rm C}_{\rm b}(\R^d)$ is differentiable with $\varphi(t)\in D$ for all $t>0$ if, e.g., $\varphi\in {\rm C}_{\rm b}^{1,2}\big((0,\infty)\times \R^d\big)$. Indeed, let $\varphi\in {\rm C}_{\rm b}^{1,2}\big((0,\infty)\times \R^d\big)$. Then, by definition, $\varphi(t)\in {\rm C}_{\rm b}^2(\R^d)$ for all $t>0$. It remains to show that $\varphi\colon (0,\infty)\to {\rm C}_{\rm b}(\R^d)$ is differentiable, where ${\rm C}_{\rm b}(\R^d)$ is endowed with the mixed topology, cf.\ Remark \ref{rem.mixed.topology}. To that end, let $t>0$ and $(h_n)_{n\in \N}\subset \R\setminus \{0\}$ with $h_n\to 0$ as $n\to \infty$ and $t+h_n>0$ for all $n\in \N$. Then, by Taylor's theorem, it follows that
\[
 \bigg|\frac{\varphi(t+h_n,x)-\varphi(t,x)}{h_n}\bigg|\leq \sup_{s>0}|\partial_t \varphi(s,x)|
\]
for all $x\in \R^d$, where $\partial_t \varphi$ denotes the partial derivative in the time variable. Hence,
\[
\sup_{n\in \N} \bigg\|\frac{\varphi(t+h_n)-\varphi(t)}{h_n}\bigg\|_\infty\leq \|\partial_t \varphi(s,x)\|_\infty.
\]
It remains to prove that the difference quotient $\frac{\varphi(t+h_n,x)-\varphi(t,x)}{h_n}$ converges to $\partial_t\varphi(t,x)$ uniformly on space compacts or, equivalently, 
\[
\frac{\varphi(t+h_n,x_n)-\varphi(t,x_n)}{h_n}\to \partial_t \varphi(t,x)
\]
for any sequence $(x_n)_{n\in \N}\subset \R^d$ with $x_n\to x\in \R^d$ as $n\to \infty$. To that end, let $(x_n)_{n\in \N}\subset \R^d$ with $x_n\to x\in \R^d$ as $n\to \infty$. Then, since $\partial_t\varphi \colon (0,\infty)\times \R^d\to \R$ is continuous, by the intermediate value theorem and the fundamental theorem of calculus, it follows that
\[
 \frac{\varphi(t+h_n,x_n)-\varphi(t,x_n)}{h_n}=\frac1{h_n}\int_t^{t+h_n} \partial_t \varphi(s,x_n)\,\dd s\to \partial_t \varphi(t,x) \quad \text{as }n\to \infty.
\]
Hence, the standard class of test functions from viscosity theory is contained in our class of test functions. In particular, in this situation, uniqueness of viscosity solutions can be concluded from classical comparison principles, see, e.g.,\ \cite{CIL92}.
\end{remark}

\section{Main result}\label{sec:main_result}

We are now ready to state and prove our main result, which is a generalisation of \cite[Theorem 6.2]{GoNeRo24} for convex monotone semigroups in the mixed topology, \cite[Theorem 4.5]{NeRo21} for upper envelopes of Feller semigroups, and \cite[Proposition 5.11]{DeKuNe20} in the context of sublinear convolution semigroups.

\begin{theorem}\label{thm:existence}
    Let $S=(S(t))_{t\geq 0}$ be a $K$-convex monotone semigroup with $S(t)0\geq0$ for all $t\geq0$, generator $L\colon D(L)\subset X\to Y$, and a strongly right-continuous family $K=(K(t))_{t\geq 0}$.\ Moreover, let $M\subset Y_+'$ and $D\subset D(L)$ be nonempty.\ Then, for every $u_0 \in X$, the function $u \colon [0, \infty) \to X, \; t \mapsto S(t)u_0$ is a $D$-$M$-viscosity solution to the abstract initial value problem
    \begin{align*}
        u'(t) &\, =\, Lu(t) \quad \text{for all } t>0,\\
        u(0) \, & \,=u_0.
    \end{align*}
\end{theorem}

\begin{proof}
    Fix $t > 0$ and $\mu \in M$. In a first step, we show that $u$ is a  subsolution. To that end, 
   let $\mu\in M$ and $\varphi \colon (0, \infty) \to X$ be differentiable with $\varphi(t) \in D$, $\mu\varphi(t) = \mu u(t)$, and $\varphi(s) \geq u(s)$ for all $s >0$. Then, using the semigroup property, for $h \in (0, 1)$ with $h<t$, we get
    \begin{align*}
        0 &= \frac{S(h)S(t-h) u_0 - S(t)u_0}{h} = \frac{S(h)u(t-h) - u(t)}{h} \leq \frac{S(h)\varphi(t-h) - u(t)}{h}\\
        &= \frac{S(h)\varphi(t-h) - S(h)\varphi(t)}{h} + \frac{S(h)\varphi(t) - \varphi(t)}{h} + \frac{\varphi(t) - u(t)}{h}\\
        &\leq K(h) \bigg( \varphi(t) + \frac{\varphi(t-h) - \varphi(t)}{h} \bigg) - S(h)\varphi(t) + \frac{S(h)\varphi(t) - \varphi(t)}{h} + \frac{\varphi(t) - u(t)}{h},
    \end{align*}
    where the last inequality follows from Lemma \ref{lem.difference-quotient} with $u=\varphi(t-h)$ and $v=\varphi(t)$.

    Using the strong right-continuity of $K$ and Lemma \ref{lem.right-cont-S}, we find that
    \begin{equation*}
         K(h) \left( \varphi(t) + \frac{\varphi(t-h) - \varphi(t)}{h} \right) - S(h)\varphi(t) \rightarrow -\varphi'(t)\quad \text{w.r.t.\ }\tau_Y\text{ as }h\downarrow 0.
    \end{equation*}
   On the other hand, since $\varphi(t)\in D\subset D(L)$, it follows that
    \begin{equation*}
        \frac{S(h)\varphi(t) - \varphi(t)}{h} \rightarrow L\varphi(t)\quad \text{w.r.t.\ }\tau_Y\text{ as }h\downarrow 0.
    \end{equation*}
    Since $\mu\in Y_+'$ and $\mu\varphi(t) = \mu u(t)$, we can thus conclude
    \begin{align*}
        0 &\leq \lim_{h\downarrow 0}
        \mu\Bigg(K(h) \bigg( \varphi(t) + \frac{\varphi(t-h) - \varphi(t)}{h} \bigg) - S(h)\varphi(t) \Bigg) + \lim_{h\downarrow 0}\mu\bigg(\frac{S(h)\varphi(t) - \varphi(t)}{h}\bigg)\\
        &
        =\mu(-\varphi'(t)) + \mu L\varphi(t)=-\mu \varphi'(t)+\mu L\varphi(t).
    \end{align*}
    Hence, $\mu \varphi'(t)\leq \mu L\varphi(t)$.\
    Next, we show that $u$ is a supersolution. Now, let $\mu\in M$ and $\varphi \colon (0, \infty) \to X$ be differentiable with $\varphi(t) \in D$, $\mu\varphi(t) = \mu u(t)$, and $\varphi(s) \leq u(s)$ for all $s >0$. Then, for $h \in (0, 1)$ with $h<t$, the semigroup property implies that
    \begin{align*}
        0 &= \frac{S(t)u_0 - S(h)S(t-h) u_0 }{h} = \frac{u(t) - S(h)u(t-h)}{h} \leq \frac{u(t) - S(h)\varphi(t-h)}{h} \\
        &= \frac{u(t) - \varphi(t)}{h} + \frac{\varphi(t) - S(h)\varphi(t)}{h} + \frac{S(h)\varphi(t) - S(h)\varphi(t-h)}{h}\\
        &\leq \frac{u(t) - S(h)\varphi(t-h)}{h} + \frac{\varphi(t) - S(h)\varphi(t)}{h}\\
        &\quad + K(h)\bigg( \varphi(t - h) + \frac{\varphi(t) - \varphi(t - h)}{h} \bigg) - S(h)\varphi(t - h),
    \end{align*}
    where, in the last inequality, we used Lemma \ref{lem.difference-quotient} with $u=\varphi(t)$ and $v=\varphi(t-h)$. Since $\varphi(t)\in D\subset D(L)$, it follows that  
    \begin{equation*}
        \frac{\varphi(t) - S(h)\varphi(t)}{h} \to -L\varphi(t)\quad\text{w.r.t.\ }\tau_Y\text{ as }h\downarrow 0.
    \end{equation*}
    Using the fact that the family $K$ is strongly right-continuous together with Lemma \ref{lem.right-cont-S}, we find that
    \begin{equation*}
       K(h)\bigg( \varphi(t - h) + \frac{\varphi(t) - \varphi(t - h)}{h} \bigg) - S(h)\varphi(t - h)  \to \varphi'(t)\quad \text{w.r.t.\ }\tau_Y\text{ as }h\downarrow 0.
    \end{equation*}
    Since $\mu\in Y_+'$ with $\mu\varphi(t) = \mu u(t)$, we can thus conclude
    \begin{align*}
        0 & \leq \lim_{h\downarrow 0} \mu\bigg(\frac{\varphi(t)-S(h)\varphi(t)}{h}\bigg)+\lim_{h\downarrow 0}\mu\Bigg( K(h)\bigg( \varphi(t - h) + \frac{\varphi(t) - \varphi(t - h)}{h} \bigg)\\
        &\quad - S(h)\varphi(t - h) \Bigg)= \mu\big(-L\varphi(t)\big) + \mu \varphi'(t)= -\mu L\varphi(t)+ \mu \varphi'(t).
    \end{align*}
    Hence, $\mu L\varphi(t)\leq \mu\varphi'(t)$. We have therefore shown that $u$ a $D$-$M$-viscosity solution.
\end{proof}

\section{Examples}\label{sec:examples}

In this section, we illustrate that our class of nonlinear semigroups and thus our main result, Theorem \ref{thm:existence}, comprises a large class of value functions of optimal control problems under uncertainty. Hence, they can be characterized as viscosity solutions to the abstract Cauchy problem given in terms of their generators.

In Sections \ref{sec:levy} and \ref{sec:OU}, we treat a drift-controlled infinite-dimensional L\'evy process and drift-controlled infinite-dimensional Ornstein-Uhlenbeck process under volatility uncertainty, respectively.

\subsection{Infinite dimensional control problem with L\'evy dynamics}\label{sec:levy}

Let $H$ be a separable Hilbert space and $(L_t)_{t\geq0}$ be an $H$-valued L\'evy process on a complete filtered probability space $(\Omega,\mathcal F, (\mathcal F_t)_{t\geq0},\P)$, satisfying the usual conditions. We consider the set $\mathcal A$ of all $H$-valued adapted stochstic processes $(\alpha_t)_{t\geq0}$ with $\int_0^t \|\alpha_s\|\,\dd s<\infty $ $\P$-a.s.\ for all $t\geq0$. For $\alpha\in \mathcal A$, $t\geq0$, and $x\in H$ we define the controlled dynamics as
\[
X_t^{x,\alpha}:=x+\int_0^t \alpha_s\,\dd s+L_t.
\]
Moreover, we fix a running cost $c\colon H\to [0,\infty]$ with $c(0)=0$ and $\lim_{\|a\|\to \infty} \frac{c(a)}{\|a\|}=\infty$, and consider the value function
\[
\big(S(t)u\big)(x):=\sup_{\alpha\in \mathcal A} \E\bigg(u\big(X_t^{x,\alpha}\big)-\int_0^t c(\alpha_s)\dd s\bigg)
\]
for all $u\in X:=\UCb(H)$, $t\geq0$, and $x\in H$, where $\UCb(H)$ denotes the space of all bounded uniformly continuous function $H\to \bR$. Moreover, let $Y$ be the space of all functions $H\to \R$, endowed with the topology of pointwise convergence and $M$ be the set of all Dirac measures with barycenters in $H$.

Let $t\geq0$, $u\in \UCb(H)$, and $\eps >0$. Then, there exists some $\delta>0$ such  that
\[
|u(x)-u(y)|<\eps\quad\text{for all }x,y\in H\text{ with }\|x-y\|<\delta.
\]
Then,
\[
\big|\big(S(t)u\big)(x)-\big(S(t)u\big)(y)\big|\leq \sup_{\alpha\in \mathcal A}|\E\big(u(X_t^{x,\alpha})-u(X_t^{y,\alpha})\big)|<\eps
\]
for all $x,y\in H$ with $\|x-y\|<\delta$, since $X_t^{x,\alpha}-X_t^{y,\alpha}=x-y$ for all $x,y\in H$ and $\alpha\in \mathcal A$. Moreover, $\|S(t)u\|_\infty\leq \|u\|_\infty$. We have therefore shown that $S(t)\colon \UCb(H)\to \UCb(H)$ is well-defined. By definition, $S(0)u=u$ for all $u\in \UCb(H)$ and one readily verifies that $S(t)$ is monotone for all $t\geq0$.

The semigroup property follows from the dynamic programming principle, see, e.g., \cite{KuNeSg25} for an elementary proof with convex $c$.

By Remark \ref{rem.Kconvex} a), as $S(t)$ is convex, we may choose $K(t):=S(t)$ for all $t\geq0$, so that $S$ is $K$-convex.

Next, we show that $K$ is strongly right-continuous.\ Since $\|S(t)u-S(t)v\|_\infty\leq \|u-v\|_\infty$ for all $f,g\in \UCb(H)$ and the space ${\rm Lip}_{\rm b}(H)$ of all bounded Lipschitz continuous functions $H\to \bR$ is dense in $\UCb(H)$, it suffices to show that
\[
\big|\big(S(t)u\big)(x)-u(x)\big|\to 0\quad\text{as }t\downarrow 0
\]
for all $u\in {\rm Lip}_{\rm b}(H)$ and $x\in X$.
To that end, first observe that, by assumption
\[
\overline c(M):=\sup_{a\in H} \big(M\|a\|-c(a)\big)<\infty\quad \text{for all }M\geq0.
\]
Hence, for all $u\in {\rm Lip}_{\rm b}(H)$ and $x\in X$,
\[
\big(S(t)u\big)(x)-\E \big(u(x+L_t)\big)\leq \sup_{\alpha\in \mathcal A} \E\bigg( M\int_0^t\|\alpha_s\|\,\dd s -\int_0^t c(\alpha_s)\,\dd s\bigg)\leq t \overline c\big(\|u\|_{\rm Lip}\big),
\]
where $\|u\|_{\rm Lip}$ denotes the smallest Lipschitz constant of a function $u\in \text{Lip}_{\rm b}(H)$. Moreover, choosing $\alpha_t\equiv 0$ for all $t\geq0$, we find
\[
\E \big(u(x+L_t)\big)-\big(S(t)u\big)(x)\leq 0.
\]
Now, the claim follows from the well-known strong continuity of the transition semigroup of the L\'evy process $(L_t)_{t\geq0}$ on $\UCb(H)$.

In a last step, we prove that, for all $f\in D:=\UCb^2(H)$ and $x\in H$, 
\[
\lim_{t\downarrow 0}\frac{\big(S(t)u\big)(x)-u(x)}{t}=(A u)(x)+c^*\big(\nabla u(x)\big),
\]
where $\UCb^2(H)$ denotes the space of all twice differentiable functions with bounded and uniformly continuous derivatives,
\[
 c^*(b):=\sup_{a\in H} \big(ba-c(a)\big)\quad\text{for all }b\in H,
\]
and 
\[
(Au)(x):=\lim_{t\downarrow0} \frac{\E\big(u(x+L_t)\big)-u(x)}{t}
\]
is the restriction of the infinitesimal generator of the L\'evy process $(L_t)_{t\geq0}$ to the space $\UCb^2(H)$.

By the L\'evy-Khinchin formula, cf.\ \cite[Theorem 5.7.3]{Li86},  
\[
(Au)(x)=\lim_{t\downarrow0} \frac{\E\big(u(x+L_t)\big)-u(x)}{t}
\]
exists for all $x\in H$ and $u\in \UCb^2(H)$. It therefore remains to show that
\[
\lim_{t\downarrow 0}\frac{\big(S(t)u\big)(x)-\E\big(u(x+L_t)\big)}{t}=c^*\big(\nabla u(x)\big)
\]
for all $x\in H$ and $u\in \UCb^2(H)$. 

To that end, first observe that, for $t\geq0$, $u \in \text{Lip}_{\rm b}(H)$, and $\alpha^*\in \mathcal A$ with
\[
    \bE \big(u(x + L_t)\big) \leq \sup_{\alpha \in \cA} \bE \bigg(u(X_t^{x,\alpha}) - \int_0^t c(\alpha_s) \dd s \bigg) \leq t+\bE \bigg(u(X_t^{x,\alpha^*}) - \int_0^t c(\alpha_s^*) \dd s \bigg),
\]
it follows that
\begin{equation}\label{eq:levy_ex:control_up}
    \bE \bigg( \frac1t \int_0^t c(\alpha^*_s) \dd s \bigg) \leq 1+\|u\|_{\text{Lip }} \bE\bigg(\frac1t \bigg\| \int_0^t \alpha^*_s \,\dd s\bigg\|\bigg).
\end{equation}
Additionally, using Jensen's inequality and the triangle inequality, we find
\begin{equation}\label{eq:levy_ex:control_down}
    \bE \bigg( \frac1t \int_0^t c(\alpha_s) \dd s \bigg) \geq \bE \bigg( \frac1t \int_0^t \cbarbar(\|\alpha_s\|) \dd s \bigg) \geq \cbarbar \Bigg(  \bE\bigg(\frac1t\bigg\|\int_0^t \alpha^*_s \dd s\bigg\| \bigg)\Bigg),
\end{equation}
where
\[
    \cbar (v) \coloneqq \sup_{a\in H}\big( v \|a\| - c(a)\big)\in [0,\infty)\quad\text{and}\quad \cbarbar (v) \coloneqq \sup_{w\geq0} \big(v w - \cbar(w)\big)\in [0,\infty]\quad\text{for all }v\geq0.
\]
Note that  $\cbarbar$ is convex with $\cbarbar(\|a\|) \leq c(a)$ for all $a\in H$.
Furthermore, $\cbarbar$ grows superlinearly, i.e., for $v\geq0$,
\[
    \frac{\cbarbar(v)}{v} = \sup_{w\geq0} \bigg(\frac{vw}{v} - \frac{\cbar(w)}v\bigg)=\sup_{w\geq0} \bigg(w - \frac{\cbar(w)}v\bigg) \rightarrow \infty\quad \text{as }v\to \infty.
\]
Taking together \eqref{eq:levy_ex:control_up} and \eqref{eq:levy_ex:control_down} implies that any $t$-optimal control $\alpha^*$ satisfies
\begin{equation}\label{eq:levy_ex:control_bound}
     \bE\bigg(\frac1t\bigg\|\int_0^t \alpha^*_s\, \dd s\bigg\|\bigg) \leq M_1,
\end{equation}
where the constant $M_1\geq0$ only depends on the running cost $c$ and $\|u\|_{\text{Lip}}$.\ This, in turn, implies that, for any $u \in \UCb^2(H)$, we have
\begin{align*}
    \frac{S(t)u(x) - \bE(x+L_t)}t &= \sup_{\alpha \in \cA_t^*} \frac1t \bE \bigg( \int_0^t \int_0^1 \nabla u\bigg(x+L_t+v\int_0^t\alpha_\nu \dd\nu\bigg)\alpha_s - c(\alpha_s)\,\dd v\,\dd s\bigg)\\
    &\leq \sup_{\alpha \in \cA_t^*} \bE \bigg(\int_0^1 c^*\bigg( \nabla u\bigg(x+L_t+v\int_0^t\alpha_\nu \dd\nu\bigg)\,\dd v \bigg)\\
    &\leq \sup_{\alpha \in \cA_t^*} \underbrace{\bE \Big(c^*\big(\nabla u (x+L_t)\big)\Big)}_{\xrightarrow{t\downarrow 0}\, c^*(\nabla u(x))} + \underbrace{\frac12 \|c^*\circ \nabla u\|_{\text{Lip }} \bE\bigg(\bigg\|\int_0^t \alpha_s \dd s\bigg\|\bigg)}_{\xrightarrow{t \downarrow 0}\, 0},
\end{align*}
where $\cA_t^* \subset \cA$ is the set of admissible controls satisfying \eqref{eq:levy_ex:control_bound}, the first term converges uniformly in $x\in H$ since $u \in \UCb^2(H)$ and $(L_t)_{t\geq0}$ is a L\'evy process, and the second term converges due to the boundedness of $\nabla u$, which implies that $c^*\circ \nabla u$ is Lipschitz continuous, and \eqref{eq:levy_ex:control_bound}.

On the other hand, considering constant controls and a $1$-optimal control $a^*\in H$ for $b\in H$ with $\|b\|\leq \|\nabla u\|_\infty$,

\[
    0 \leq \sup_{a \in A} \big(ba-c(a)\big) \leq 1 + ba^* - c(a^*).
\]
This, in turn, implies that $\|a^*\|\leq M_2$ with a constant $M_2\geq0$ depending only on the running cost $c$ and $\|\nabla u\|_\infty$. We then find that
\begin{align*}
    \frac{S(t)u(x) - \bE\big(u(x+L_t)\big)}t &\geq \sup_{\|a\|\leq M_2} \Bigg(\bE \bigg(\int_0^1 \nabla u(x+L_t+ vta)\, \dd v \bigg) a - c(a)\Bigg)\\
    &\geq \sup_{\|a\|\leq M_2} \Bigg(\bE \big(\nabla u(x + L_t)\big) a - c(a) - \frac12t \|a\|^2 \|\nabla u\|_{\text{Lip}}\Bigg)\\
    &=c^*\Big(\bE\big(\nabla u(x+L_t)\big)\Big)-\frac12 tM_2^2\|\nabla u\|_{\rm Lip}.
\end{align*}
We can therefore conclude that
\[
\lim_{t\downarrow0}\frac{S(t)u(x) - \bE(x+L_t)}t = c^*\big(\nabla u(x)\big)\quad \text{uniformly in }x\in H.
\]

\subsection{Distributionally robust optimal control problems with O-U dynamics}\label{sec:OU}

Again, let $H$ be a separable Hilbert space and $B=(B_t)_{t\geq 0}$ be a $Q$-Brownian Motion with trace class covariance operator $Q\colon H\to H$ on a complete filtered probability space $(\Omega,\mathcal F, (\mathcal F_t)_{t\geq0},\P)$, satisfying the usual assumptions.

Throughout, we consider a nonempty action set $A$, which is assumed to be a Borel space and a bounded nonempty subset $\Sigma\subset L(H)$, where $L(H)$ denotes the space of all bounded linear operators $H\to H$, and define the sets of admissible controls $\mathcal{A}$ and $\Theta$ as the set of all progressively measurable processes $\alpha\colon \Omega \times [0,\infty) \to A$ and $\theta \colon \Omega\times [0,\infty) \rightarrow \Sigma$, respectively, with
$\int_0^t\|\alpha_s\|\,\dd s<\infty$ $\P$-a.s.\ for all $t\geq0$.

Moreover, we consider a measurable function $b\colon H\times A\to H$ satisfying
\begin{align*}\label{eq:lip_growth_b_sig}
        & \|b(0, a)\| \leq C \quad \text{for all }a\in A,\\
	&\|b(x_1,a)-b(x_2,a)\|  \leq C\|x_1-x_2\|\quad \text{for all }a\in A\text{ and }x_1,x_2\in H.
\end{align*}

Then, for all initial values $x\in H$ and all admissible controls $\alpha\in \cA$ and $\theta \in \Theta$, we consider the controlled SPDE
    \begin{align*}
     \dd X_t^{x,\alpha, \theta} &= \Big(AX_{t}^{x,\alpha,\theta}+b\big(X_t^{x,\alpha, \theta}, \alpha_t\big)\Big)\,\dd t +\theta_t\,\dd B_t\quad \text{for } t\geq 0,\\
    X_0^{x,\alpha, \theta} &=x,
    \end{align*}
 where $A\colon D(A)\subset H\to H$ is the generator of a $C_0$-semigroup $(P_t)_{t\geq0}$ on $H$.\ Then, for all $x\in H$, $\alpha\in \cA$, and $\theta \in \Theta$, a  mild solution to the controlled SPDE is given by the implicit equation
\[
 X_t^{x,\alpha,\theta}=P_tx+\int_0^t P_{t-s} b\big(X_s^{x,\alpha,\theta},\alpha_s\big) \,\dd s+ \int_0^t P_{t-s} \theta_s\,\dd B_s\quad \text{for all }t\geq0.
\]
Again, we consider the space $X:={\rm UC}_{\rm b}(H)$, endowed with the supremum norm, and the space $Y$ of all functions $H\to \R$, endowed with the topology of pointwise convergence. As before, we choose $M$ to be the set of all Dirac measures with barycenters in $H$. For $u\in {\rm UC}_{\rm b}(H)$, we then consider the value function of the related Meyer-type stochastic optimal control problem 
\[
\big(S(t)u\big)(x):=\inf_{\theta \in \Theta}\sup_{\alpha\in \mathcal A}\mathbb E\big[u\big(X_t^{x,\alpha,\theta}\big)\big].
\]
Clearly $S(0)u=u$ for all $u\in {\rm UC}_{\rm b}(H)$. We first show that $S(t)\colon {\rm UC}_{\rm b}(H)\to {\rm UC}_{\rm b}(H)$ is well-defined for all $t\geq0$. This, however, is a direct consequence from the estimate
\[
e^{-\omega t}\big\|X_t^{x,\alpha,\theta}-X_t^{y,\alpha,\theta}\big\|\leq \|x-y\|+C \int_0^t e^{-\omega s}\big\|X_s^{x,\alpha,\theta}-X_s^{x,\alpha,\theta}\big\|\,\dd s\quad\P\text{-a.s.}
\]
for all $t\geq0$, $\alpha\in \mathcal A$, $\theta\in \Theta$, and $x,y\in H$ together with Gronwall's inequality.
The semigroup property follows again from the dynamic programming principle, see, e.g., \cite[Theorem 2.24, p. 110]{FaGoSw17}, for the case without additional volatility control.

For $t\geq0$, $u\in {\rm UC}_{\rm b}(H)$, and $x\in H$, let
\[
\big( K(t)u\big)(x):=\sup_{\theta \in \Theta}\sup_{\alpha\in \mathcal A}\mathbb E\big[u\big(X_t^{x,\alpha,\theta}\big)\big].
\]
Let $(t_n)_{n\in \N}\subset (0,\infty)$ with $t_n\to 0$ as $n\to \infty$ and $(u_n)_{n\in\N}\subset  {\rm UC}_{\rm b}(H)$ with $u_n\to u\in {\rm UC}_{\rm b}(H)$ uniformly as $n\to \infty$. We prove that $K(t_n)u_n\to u$. First, observe that $\|K(t_n)u_n-K(t_n)u\|_\infty\leq \|u_n-u\|_\infty$ for all $n\in \N$. It therefore remains to show that $\big(K(t_n)u\big)(x)\to u(x)$ for all $x\in H$. Since ${\rm Lip}_{\rm b}(H)$ is dense in ${\rm UC}_{\rm b}(H)$, we may additionally assume that $u\in {\rm Lip}_{\rm b}(H)$. Then, the claim follows from the estimate
\begin{align*}
\E\big[\big\| X_t^{x,\alpha,\theta} -x\|\big]&\leq \|P_t x-x\|+ \int_0^t e^{\omega(t-s)}\big(C(1+\|x\|) +\E\big[\big\| X_t^{x,\alpha,\theta} -x\|\big]\big)\,\dd s\\
&\quad +\bigg(\int_0^t e^{-2\omega (t-s)}\|\theta_s\|^2_{L(H)}{\rm tr} (Q)\,\dd s\bigg)^{1/2}
\end{align*}
for all $t\geq 0$, $x\in H$, $\alpha\in \mathcal A$, and $\theta\in \Theta$ together with Gronwall's lemma.

Using It\^o's lemma, the fact that $b$ is uniformly Lipschitz in the action $a\in A$ and $\Sigma\subset L(H)$ is bounded, it follows that
\[
\lim_{t\downarrow 0}\frac{\big(S(t)u\big)(x)-u(x)}{t}= \langle x,A^*\nabla u(x)\rangle+\sup_{a\in A}\langle b(x,a),\nabla u(x)\rangle+\inf_{\sigma\in \Sigma} \frac12 {\rm tr}\big(\sigma^*Q\sigma\nabla^2u(x)\big)
\]
for all $u\in {\rm UC}_{\rm b}^2(H)$ with $\nabla u(x)\in D(A^*)$ for all $x\in H$ and $A^*\nabla u\in {\rm UC}_{\rm b}(H)$.

\printbibliography

\end{document}